\documentclass[letterpaper, 10 pt, conference]{ieeeconf}
\IEEEoverridecommandlockouts  %
\usepackage[utf8]{inputenc}

\usepackage[T1]{fontenc}

\usepackage[T1]{fontenc}    %
\usepackage[hidelinks]{hyperref}       %
\usepackage{url}            %
\usepackage{booktabs}       %
\usepackage{amsfonts}       %
\usepackage{nicefrac}       %
\usepackage{xspace}

\usepackage{siunitx}
\usepackage{bm}
\usepackage{comment}
\usepackage{amsmath}
\usepackage{amsthm}
\usepackage{graphicx}
\usepackage{rotating}
\usepackage{float}
\usepackage{mathrsfs}
\usepackage{amssymb}
\usepackage{autobreak}
\usepackage{mathtools}
\usepackage{wrapfig}
\usepackage{xcolor}
\usepackage{bbm}
\usepackage[ruled,linesnumbered]{algorithm2e} %
\usepackage[]{algpseudocode}

\usepackage{subcaption}
\usepackage{tabularx}

\usepackage{lipsum}
\usepackage{outlines}
\usepackage{blindtext}
\usepackage{multicol}
\usepackage[capitalise]{cleveref}
\usepackage{tikz}
\usetikzlibrary{shapes,backgrounds}
\usepackage{array}
\usepackage[export]{adjustbox}
\newsavebox{\algleft}
\newsavebox{\figright}
\crefname{assumption}{Assumption}{Assumptions}

\usepackage{cite}
\usepackage{setspace}

\usepackage{doi}

\newtheorem{theorem}{Theorem}
\newtheorem{lemma}[theorem]{Lemma}
\newtheorem{proposition}[theorem]{Proposition}
\newtheorem{corollary}[theorem]{Corollary}
\newtheorem{remark}{Remark}
\newtheorem{assumption}{Assumption}

\newcommand{\coninput}{u}
\newcommand{\state}{x}
\newcommand{\gpinp}{z}
\newcommand{\dynGP}{g}
\newcommand{\dynGPtrue}{g^{\mathrm{tr}}}
\newcommand{\dynNom}{f}
\newcommand{\statedim}{{n_\state}}
\newcommand{\inputdim}{{n_\coninput}}
\newcommand{\gpdim}{{n_\dynGP}}
\newcommand{\datasetdim}{D}
\newcommand{\testinputdim}{n_\star}
\newcommand{\horizon}{H}
\newcommand{\numsqp}{L} 
\newcommand{\numsamples}{N}
\newcommand{\prob}{p}
\newcommand{\dynSet}[1][]{\mathcal{G}_{#1}}
\newcommand{\n}{n}
\newcommand{\hessian}{H}
\newcommand{\kernelfunc}{k}

\newcommand{\kernelfuncjoint}{k_{\mathrm{d}}}
\newcommand{\mujoint}{\mu_{\mathrm{d}}}
\newcommand{\Sigmajoint}{\Sigma_{\mathrm{d}}}
\newcommand{\gpsubspace}{B_d}
\newcommand{\Bg}{B_\dynGP}

\newcommand{\constraint}{h}
\newcommand{\gpmeasvar}{\lambda}
\newcommand{\gppostvar}{\sigma}
\newcommand{\ubdyn}[1][]{\overline{\dynGP_{#1}}}
\newcommand{\lbdyn}[1][]{\underline{\dynGP_{#1}}}
\newcommand{\GPtrunc}[1][N]{\mathcal{#1}_{[\lbdyn, \ubdyn]}}

\newcommand{\ki}{k}
\newcommand{\cost}{l}

\newcommand{\xpos}{x_p}
\newcommand{\ypos}{y_p}
\newcommand{\xposref}{x_{\mathrm{ref}}}
\newcommand{\yposref}{y_{\mathrm{ref}}}

\newcommand{\R}{\mathbb{R}}

\newcommand{\muconst}[1][]{\mu_{#1}}

\newcommand{\betaconst}[1][]{\beta_{#1}}

\begin{document}

\title{Towards safe and tractable Gaussian process-based MPC: \\ Efficient sampling within a sequential quadratic programming framework}
\author{
    Manish Prajapat$^{\star}$, Amon Lahr$^{\star}$, Johannes K\"ohler, Andreas Krause, Melanie N. Zeilinger
    \thanks{$^\star$First author, equal contribution. E-mail correspondence to: \texttt{manishp@ai.ethz.ch}; \texttt{amlahr@ethz.ch}.
    All authors are from ETH Zürich. Manish Prajapat is supported by ETH
AI center, Amon Lahr by the European Union's Horizon 2020 research and innovation programme, Marie~Sk\l{}odowska-Curie grant agreement No. 953348,~\mbox{ELO-X}, and
    Johannes  K\"ohler by the Swiss National Science Foundation under NCCR Automation, grant agreement 51NF40 180545.
}
}
\maketitle

\begin{abstract}
\looseness -1 
Learning uncertain dynamics models using Gaussian process~(GP) regression has been demonstrated to enable high-performance and 
safety-aware
control strategies for challenging real-world applications. 
Yet, for computational tractability, 
most approaches for Gaussian process-based model predictive control (GP-MPC) are based on 
approximations of the 
reachable set
that are either overly conservative or impede the controller's 
safety guarantees.
To address these challenges, 
we propose a robust GP-MPC formulation 
that guarantees constraint satisfaction 
with high probability.
For its tractable implementation, we propose a sampling-based GP-MPC approach that iteratively generates consistent dynamics samples from the GP within a sequential quadratic programming framework.
We
highlight 
the improved reachable set approximation compared to existing methods,
as well as real-time feasible computation times, 
using two numerical examples.
\end{abstract}

\section{Introduction}

Gaussian process~(GP) regression~\cite{rasmussen_gaussian_2006}
offers versatile representation capabilities and inherent uncertainty quantification
for learning dynamics models.
Combined with model predictive control~(MPC)~\cite{rawlingsModelPredictiveControl2020},
GP models
have 
demonstrated
to be highly effective 
at improving the control performance 
while ensuring safety-awareness 
with respect to the 
remaining 
model uncertainty, 
across a diverse range of challenging 
real-world 
scenarios~\cite{ostafew_robust_2016,cao_gaussian_2017,carron_data-driven_2019,kabzan_learning-based_2019,hewing_cautious_2020,vaskov_friction-adaptive_2022}. 
Yet, 
propagating the uncertainty 
within the
resulting 
Gaussian process state-space model~(GP-SSM) 
is non-trivial and
remains a major bottleneck for safety-critical GP-MPC applications~\cite{hewing_learning-based_2020}.
To remedy this issue, 
a standard approach~\cite{kabzan_learning-based_2019,hewing_cautious_2020,vaskov_friction-adaptive_2022,lahr_zero-order_2023-1}
is to
use a linearization-based approximation of the probabilistic reachable set induced by the uncertain GP model~\cite{girard_gaussian_2002-1}. 
However, 
in addition to the linearization error,
this approximation is based on a strong independence assumption between the GP predictions at subsequent time steps~\cite{hewing_simulation_2020}.
Similarly, 
approximate uncertainty propagation methods,
such as 
exact moment-matching~\cite{quinonero-candela_propagation_2003} with a spectral GP approximation~\cite{pan_prediction_2017},
or
unscented filtering~\cite{ostafew_robust_2016,quirynen_uncertainty_2021}, 
make it difficult
to 
guarantee
safety
for 
the resulting
GP-MPC formulations.

To establish 
constraint-satisfaction guarantees
for GP-MPC,
stochastic and robust 
MPC
formulations have been proposed.
In~\cite{koller_learning-based_2018},
a robust MPC approach
has been suggested to iteratively over-approximate 
the uncertain prediction
based on 
high-probability confidence regions of the GP.
Yet, 
this approach tends to be overly conservative
due to the ellipsoidal over-approximation based on global Lipschitz bounds, 
as well as robustification of the control input against
the worst-case GP estimate at each time step \emph{independently}, 
without taking into account 
that the uncertainty is of epistemic nature, 
i.e., 
that the true function does not change over the prediction horizon. 
To remedy 
this
conservatism,
\cite{maddalena2021kpc} employs a Gaussian-process multi-step predictor; 
however,
learning a multi-step predictor constitutes a more nonlinear and higher-dimensional regression problem,
with stronger computational requirements,
compared to
single-step models.

\looseness-1 For stochastic nonlinear MPC, 
it has been proposed 
to 
sample 
controlled
trajectories of the GP-SSM offline to determine fixed constraint tightenings
for online use, 
ensuring the desired probability of closed-loop chance-constraint satisfaction~\cite{bradford_stochastic_2020,bradford_hybrid_2021}.
Therein, 
a 
``forward-sampling''
~\cite{conti_gaussian_2009,umlauft_scenario-based_2018-1} 
strategy, 
based on reconditioning
the GP on previously sampled values,
has been employed 
to overcome the difficulty of sampling continuous functions from the GP directly,
which has also been used in~\cite{hewing_simulation_2020,lederer_confidence_2020, beckers_prediction_2021} for trajectory predictions with GPs.
However, 
the guarantees only hold for a fixed (distribution of) initial conditions and controller parameters.

To summarize,
existing
approximations of the 
reachable set associated with uncertain GP models 
either 
do not ensure
closed-loop constraint satisfaction,
or
are too conservative to be practically useful.

\subsubsection*{Contributions}
Toward addressing these challenges, 
we propose a GP-based 
MPC 
formulation that is robust against all dynamics in a confidence set, 
guaranteeing 
constraint satisfaction for the true system 
with high probability (Sec.~\ref{sec:GPMPC_prob}).

Additionally, we present an efficient sampling-based algorithm to approximately solve the derived robust GP-MPC problem based on sequential quadratic programming~(SQP). 
The proposed strategy is equivalent to solving an MPC problem with multiple dynamics sampled directly from the GP posterior, leading to a more accurate uncertainty propagation compared to \cite{koller_learning-based_2018,hewing_cautious_2020}.
Implementation is non-trivial since sampling the exact continuous function from GP is computationally intractable \cite{lin2024sampling}. 
To tackle this, we jointly model the unknown dynamics and its Jacobian using a multivariate GP and construct \emph{consistent} dynamics samples within the SQP framework by forward-sampling~\cite{hewing_simulation_2020, umlauft_scenario-based_2018-1,conti_gaussian_2009,lederer_confidence_2020, beckers_prediction_2021} across SQP iterations. 
The sampling process is parallelized across all samples and GP dimensions, resulting in efficient run time on GPUs. 
We discuss how to apply the algorithm efficiently in receding horizon using the real-time iteration~\cite{diehl_real-time_2005}, 
ensuring a maintainable computational footprint and similarity of the sampled dynamics across MPC iterations (\cref{sec:samplingGPMPC}). 

We highlight the effectiveness of sampling-based uncertainty propagation in comparison with sequential over-approximation~\cite{koller_learning-based_2018} and linearization-based~\cite{hewing_cautious_2020} approaches on a pendulum example. Lastly, we further demonstrate our method with a more realistic car model (\cref{sec:simulations}).

\section{Problem statement}

\looseness -1 We consider the task of state-feedback control for a discrete-time, nonlinear dynamical system 
\begin{align}
    \state(\ki+1) = \dynNom(\state(\ki), \coninput(\ki)) + \gpsubspace \dynGPtrue(\state(\ki), \coninput(\ki)) \label{eq:system_dyn}
\end{align}
with state 
$\state(\ki) \in \R^\statedim$ and input 
$\coninput(\ki) \in \R^\inputdim$. 
Here, 
\mbox{$\dynNom: \mathbb{R}^{\statedim} \times \mathbb{R}^{\inputdim} \rightarrow \mathbb{R}^{\statedim}$}
denotes the known part of the true dynamics, e.g., 
derived from first principles, 
and 
\mbox{$\dynGPtrue: \mathbb{R}^{\statedim} \times \mathbb{R}^{\inputdim} \rightarrow \mathbb{R}^{\gpdim}$},
the unknown part to be estimated from data,
modeled in a subspace defined by
\mbox{$\gpsubspace \in \mathbb{R}^{\statedim \times \gpdim}$}
with full column-rank.
The data consists of 
$\datasetdim \in \mathbb{N}$ 
noisy
measurements of the one-step prediction error
\begin{align}
    \label{eq:measurements}
    y_k \doteq \gpsubspace^\dagger (\state(\ki+1) - \dynNom(\state(\ki),\coninput(\ki))) + \epsilon(\ki),
\end{align}
where 
$\epsilon(\ki)$ 
is 
independent and identically distributed (i.i.d.) Gaussian measurement noise with covariance 
\mbox{$\Lambda \doteq \gpmeasvar^2 I_{n_g}$},
\mbox{$\lambda > 0$},
and $(\cdot)^\dagger$ denotes the Moore-Penrose pseudo-inverse.
For notational simplicity, and without loss of generality, the results in this paper are presented for a scalar unknown dynamics component,
i.e., 
\mbox{$\gpdim = 1$}; 
see~\cref{remark:multidim_gp} 
for an 
extension
to the multivariate case.

The main objective is 
the design of a model predictive control~(MPC) strategy
that ensures 
satisfaction of
safety constraints
\begin{align}
    (x(k), u(k)) \in \mathcal{Z} \doteq \left\{ (x,u) \> \middle| \>  \constraint(\state, \coninput) \leq 0 \right\}
    \label{eq:constraints_path}
\end{align}
at all times
{$k \in \mathbb{N}$},
despite the 
epistemic 
uncertainty 
in 
$\dynGP$.

We make the following 
regularity assumption. 

\begin{assumption}[Regularity \cite{chowdhury_kernelized_2017}]
    \label{assump:q_RKHS}
    The uncertain dynamics~$\dynGPtrue$
    is an element of the
    Reproducing Kernel Hilbert Space~(RKHS)~$\mathcal{H}_k$
    associated with the continuous, positive definite kernel 
    \mbox{$\kernelfunc: 
    \mathbb{R}^{n_z} \times \mathbb{R}^{n_z} 
    \rightarrow \mathbb{R}$}, 
    with a bounded RKHS norm,
    i.e., 
    \mbox{$\dynGPtrue \in \mathcal{H}_k$ }
    with 
    \mbox{$\|\dynGPtrue\|_{\mathcal{H}_k} \leq \Bg < \infty$}. 
\end{assumption}

To obtain a data-driven model and uncertainty quantification of the unknown dynamics~$\dynGPtrue$, Gaussian process~(GP) regression is employed.
Denote by 
\mbox{$\bar{\mathcal{D}} \doteq \{ Z, Y \}$}
the
data set,
where
\mbox{$Z = \begin{bmatrix} z_1, \ldots, z_\datasetdim \end{bmatrix}$},
\mbox{$z_k \doteq (\state(\ki), \coninput(\ki)) \in \mathbb{R}^{n_z}$},
are the
training inputs
with 
\mbox{$n_\gpinp = \statedim + \inputdim$};
\mbox{$Y = \begin{bmatrix} y_1, \ldots, y_\datasetdim \end{bmatrix}$},
the 
training targets; and
\mbox{$Z_\star = \begin{bmatrix} z_{1}^\star, \ldots, z_{\testinputdim}^\star \end{bmatrix}$},
the test inputs.
Then,
the posterior mean and covariance are given as 
\begin{align}
    \mu(Z_\star) &= \kernelfunc(Z_\star, Z) 
    \tilde{K}_{ZZ}^{-1}
    Y, \\
    \gppostvar^2(Z_\star, Z'_\star) &= \kernelfunc(Z_\star, Z'_\star) - \kernelfunc(Z_\star, Z)
    \tilde{K}_{ZZ}^{-1}
    \kernelfunc(Z, Z'_\star), 
    \label{eq:GP_posterior_covariance}
\end{align} 
respectively,
where 
$\left[ \kernelfunc(Z,Z') \right]_{ij} \doteq \kernelfunc(z_i, z'_j)$ for any matrices $Z, Z'$ with respective column vectors $z_i, z'_j \in \mathbb{R}^{n_z}$,
and  
$\tilde{K}_{ZZ} \doteq k(Z,Z) + \gpmeasvar^2 I_{\datasetdim}$.

Assumption~\ref{assump:q_RKHS} 
enables
the following high-probability bounds 
for the unknown part of the dynamics,
which has also been utilized in prior work for robust GP-MPC~\cite{koller_learning-based_2018}.

\begin{lemma}
    {\upshape (\cite[Theorem 2]{chowdhury_kernelized_2017})}
    \label{thm: beta} %
    \looseness -1 
    Let 
    \cref{assump:q_RKHS} 
    hold
    and let
    $\sqrt{\betaconst} = \Bg + 4 \gpmeasvar \sqrt{\gamma_\datasetdim + 1+ \log(1/\prob)}$, 
    where 
    $\gamma_\datasetdim$ 
    is the maximum information gain as defined in~\cite[Sec.~3.1]{chowdhury_kernelized_2017}.
    Then, 
    for all
    \mbox{$\gpinp \in \mathcal{Z}$},
    with probability at least $1-\prob$,
    it holds that
    \begin{align*}
        |\dynGPtrue(\gpinp) - \muconst(\gpinp)| \leq \sqrt{\betaconst} \gppostvar (\gpinp).
    \end{align*}
\end{lemma}

Using 
\cref{thm: beta},
we can construct high-probability upper and lower confidence bounds 
\mbox{$\lbdyn(\gpinp) \doteq \mu(\gpinp) - \sqrt{\beta} \sigma(z)$},
\mbox{$\ubdyn(\gpinp) \doteq \mu(\gpinp) + \sqrt{\beta} \sigma(z)$},
respectively,
such that we can guarantee that the unknown function $\dynGPtrue$ is 
contained in the set
\begin{align}
    \dynSet \coloneqq \left\{ \dynGP \> | \> \lbdyn(\gpinp) \leq \dynGP(\gpinp) \leq \ubdyn(\gpinp) \>  
    \forall \gpinp \in \mathcal{Z} 
    \right\} 
\end{align}
with probability at least $1 - p$, i.e.,
\mbox{$\mathrm{Pr}(\dynGPtrue \in \dynSet) \geq 1-p$}.

To this end, 
we denote by
$\GPtrunc[GP](0,\kernelfunc;\bar{\mathcal{D}})$
the truncated distribution of functions 
\mbox{$g \sim \mathcal{GP}(0,\kernelfunc;\bar{\mathcal{D}})$}
that are contained
in the set 
\mbox{$g \in \dynSet$}.
For a finite number of inputs~\mbox{$Z_\star \in \mathbb{R}^{n_\gpinp \times \testinputdim}$}, 
the values~$G_\star$ of a function 
\mbox{$g \sim \GPtrunc[GP](0,\kernelfunc;\bar{\mathcal{D}})$} 
are thus distributed according
to the $\testinputdim$-dimensional \emph{truncated} Gaussian distribution
$\GPtrunc[N](G_\star; \mu(Z_\star), \Sigma(Z_\star))$
on the hyper-rectangle
\mbox{$[\lbdyn(z_1^\star), \ubdyn(z_1^\star)] \times \ldots \times [\lbdyn(z_{\testinputdim}^\star), \ubdyn(z_{\testinputdim}^\star)]$}.

\begin{remark}
    The measurement noise assumption can be relaxed to conditionally $\sigma$-sub-Gaussian noise, see, e.g.,~\cite{chowdhury_kernelized_2017}.
\end{remark}

\begin{remark}
    \label{remark:multidim_gp}
    Note that 
    \cref{assump:q_RKHS} and \cref{thm: beta}
    may easily be transferred to 
    the multidimensional case
    when
    each component of~$\dynGP$ is modeled as an independent, single-dimensional GP.
\end{remark}

\begin{remark}
    In this work, we focus on epistemic uncertainty introduced by the missing knowledge about g.
    As such, we consider noise-free true dynamics~\eqref{eq:system_dyn} with exact state measurements, 
    while allowing for measurement noise in the data set~$\bar{\mathcal{D}}$ collected offline.
    For the special case of noise-free measurements,
    tighter bounds than \cref{thm: beta} for the unknown dynamics~$\dynGP$ may be employed, see, e.g., 
    \cite[Sec.~14.1]{fasshauer_kernel-based_2015}.
\end{remark}

\section{Robust GP-MPC problem}

\label{sec:GPMPC_prob}
In this section, we formulate a robust GP-MPC problem using the derived dynamics 
distribution,
which provides safety guarantees for the unknown system. The 
proposed
MPC problem at time~$k \in \mathbb{N}$
is given by
\begin{subequations} \label{eq:SafeGPMPCinf}
    \begin{align}
        \underset{\bm{u}}{\min} \quad & 
        \mathbb{E}_{g \sim \GPtrunc[GP](0,\kernelfunc;\bar{\mathcal{D}})}
        \Bigg[ 
        \sum_{i=0}^{\horizon-1} \cost_i(\state^\dynGP_i, \coninput_i )  \Bigg] 
        \label{eq:cost}
        \\
        \mathrm{s.t.} \quad &\forall i \in \{0,\ldots,\horizon-1\}, 
        \forall g \sim \GPtrunc[GP](0,\kernelfunc;\bar{\mathcal{D}}),
        \label{eq:SafeGPMPCinf_forall} \\
        & \state^\dynGP_0 = \state(k), \label{eqn:initialization}\\
        & \state_{i+1}^\dynGP = \dynNom(\state^\dynGP_i, \coninput_i) + \gpsubspace \dynGP(\state^\dynGP_i,\coninput_i), \label{eqn:pred_dyn}\\
        & (\state^\dynGP_i,\coninput_i) \in \mathcal{Z} \label{eqn:state_input_constraint}, \\
        & \state^\dynGP_\horizon \in \mathcal{Z}_{\mathrm{f}}. \label{eqn:terminal_constraint}
    \end{align}
\end{subequations}
\looseness -1 Here, \cref{eqn:pred_dyn} yields a predicted 
state sequence
\mbox{$\bm{\state}^\dynGP = (\state^\dynGP_0, \ldots, \state^\dynGP_\horizon)$} 
for 
any
uncertain
dynamics instance
\mbox{$\dynGP \sim \GPtrunc[GP](0, \kernelfunc; \bar{\mathcal{D}})$}
under 
the same, shared
control sequence 
\mbox{$\bm{\coninput} = (\coninput_0, \ldots, \coninput_{\horizon - 1})$}
along the prediction horizon~$\horizon$.
\cref{eqn:state_input_constraint} ensures that all the predicted sequences satisfy the joint state and input constraints~\eqref{eq:constraints_path}. 
The constraint in \cref{eqn:initialization} initializes all the trajectories to the system state at time $\ki$, denoted by $\state(\ki)$, and \cref{eqn:terminal_constraint} ensures that all the predicted trajectories  
reach the
terminal set 
\begin{align}
    \mathcal{Z}_{\mathrm{f}} \doteq \left\{ \state \> \middle| \>  \constraint_\horizon(\state) \leq 0 \right\}
\end{align}
at the end of the 
horizon.
The 
cost in~\cref{eq:cost}  
reflects the
expected
finite-horizon 
cost over functions drawn from 
\mbox{$g \sim \GPtrunc[GP](0,\kernelfunc;\bar{\mathcal{D}})$}, 
with stage cost 
\mbox{$l_i: \mathbb{R}^{n_\state} \times \mathbb{R}^{n_\coninput} \rightarrow \mathbb{R}$}. 
Note that a terminal penalty 
can also be included in the cost.

Next, we investigate the theoretical guarantees of the proposed robust GP-MPC problem~\eqref{eq:SafeGPMPCinf}. In order to address constraint satisfaction for infinite time we make the following standard assumption on the terminal invariant set.
\begin{assumption}[Robust positive invariant set] \label{assum:safeset}
The terminal set $\mathcal{Z}_{\mathrm{f}}$ is a robust positive invariant set for system~\eqref{eq:system_dyn} under the input $\coninput_{\mathrm{f}}\in\mathbb{R}^m$, i.e., $\forall \dynGP \in \dynSet$, $\state \in \mathcal{Z}_{\mathrm{f}}$: \mbox{$(x,\coninput_{\mathrm{f}}) \in \mathcal{Z}$}
and \mbox{$\dynNom(\state, \coninput_{\mathrm{f}}) + \gpsubspace \dynGP(\state, \coninput_{\mathrm{f}}) \in \mathcal{Z}_{\mathrm{f}}$}.
\end{assumption}
Similar to standard 
MPC
designs~\cite[Sec.~2.5.5]{rawlingsModelPredictiveControl2020}, 
this condition can,
e.g., 
be satisfied by constructing $\mathcal{Z}_{\mathrm{f}}$ based on a local Lyapunov function around the origin.%
\footnote{%
    In particular, this requires that the origin is an equilibrium point $\forall \dynGP \in\dynSet$, i.e., $f(0,0)+B_d \dynGP(0,0)=0$, and that the linearizations at the origin are open-loop stable with a common Lyapunov function. 
    The assumption of open-loop stability can be relaxed by using a feedback $u(x)=\kappa(x)+v$ in the predictions in~\eqref{eq:SafeGPMPCinf} instead of open-loop predictions (cf., e.g.,~\cite{chisci2001systems}).%
}
The following theorem guarantees 
constraint satisfaction by applying the optimal control sequence
obtained by solving problem~\eqref{eq:SafeGPMPCinf}. 

\begin{theorem} \label{thm:closedloop_properties} 
    Let \cref{assump:q_RKHS,assum:safeset} hold.
    Let
    Problem~\eqref{eq:SafeGPMPCinf} 
    be 
    feasible at 
    $\ki=0$
    and 
    $\bm{\coninput}^\star = \{\coninput_0^\star,\dots, \coninput_{\horizon-1}^\star\}$
    be the corresponding optimal control sequence.
    Then,
    with probability at least $1-\prob$, 
    applying the control sequence
    \begin{align*}
        u(k) = \begin{cases}
            \coninput^\star_k, \quad k \in \{0, \ldots, \horizon - 1 \}\\
            u_{\mathrm{f}}, \quad k \geq \horizon
        \end{cases}
    \end{align*}
    to
    the system~\eqref{eq:system_dyn} 
    leads to constraint satisfaction for all times,
    i.e.,
    $(\state(\ki), \coninput(\ki)) \in \mathcal{Z}$ ~$\forall k\in\mathbb{N}$.
\end{theorem} 
\begin{proof} 
    First, note that $\dynGP^{\mathrm{tr}} \in \dynSet$ with probability at least $1-\prob$ due to \cref{thm: beta}. 
    Thus, for $\dynGP^{\mathrm{tr}} \in \dynSet$ and
    \mbox{$k \in \{0, \ldots, \horizon - 1 \}$},
    constraint satisfaction directly follows from feasibility of Problem \eqref{eq:SafeGPMPCinf}, 
    for $\ki \geq \horizon$, 
    from \cref{assum:safeset}.
\end{proof}

While Problem~\eqref{eq:SafeGPMPCinf} is in general not recursively feasible, closed-loop constraint satisfaction can be ensured using an implementation strategy similar to~\cite[Alg.~1]{wabersich_linear_2018}.
In the remainder of this paper,
we focus on a computationally efficient, sampling-based approximation of~\eqref{eq:SafeGPMPCinf};
further closed-loop considerations are left for future work.

\section{Efficient sampling-based GP-MPC}
\label{sec:samplingGPMPC}

Solving problem~\eqref{eq:SafeGPMPCinf} is intractable due to the infinite number of dynamics functions in the set. 
In this section,
we obtain a 
tractable 
approximation
by solving it using a finite number
of functions. 
To solve 
the nonlinear program iteratively,
we employ SQP
(\mbox{\cref{sec:SQP_basics}}).
During optimization, 
the
required
\emph{function and derivative} values 
are 
thereby 
provided by
modeling
the unknown function~$\dynGPtrue$ and its Jacobian \emph{jointly}, using a multivariate GP with a derivative kernel~(\mbox{\cref{sec:SQP_derivative_GP}}).
As it is generally not possible to obtain analytical expressions for the exact infinite-dimensional GP sample paths, 
we utilize a sequential GP sampling approach~\cite{conti_gaussian_2009,umlauft_scenario-based_2018-1}, which is based on conditioning the GP on previously sampled values, and integrate it into 
the SQP
framework~(\mbox{\cref{sec:SQP_sequential_sampling}}). 
Finally, we discuss the properties of the proposed SQP algorithm, as well as its 
efficient
closed-loop implementation~(\mbox{\cref{sec:discussion}}).

\subsection{Sequential Quadratic Programming} \label{sec:SQP_basics}

In the following, we describe how to solve~\eqref{eq:SafeGPMPCinf}
for a finite number of $N$~samples using SQP~\cite[Sec.~18]{nocedal_numerical_2006}.
Therefore, let
the functions $f,\dynGPtrue$, $h$, $h_H$, $l_i$ and $k$ be twice continuously differentiable.
A local optimizer 
can be 
obtained by 
iteratively 
approximating
the nonlinear program~(NLP)
with a quadratic program~(QP) 
\begin{subequations} \label{eq:SafeGPMPCsqp}
    \begin{align}
        \underset{\Delta \bm{\coninput}, \Delta \bm{\state}}{\min} \quad & \frac{1}{\numsamples} \sum_{n=1}^{\numsamples} \sum_{i=0}^{\horizon-1} 
        \begin{bmatrix}
            \Delta \state^\n_i \\ \Delta \coninput_i
        \end{bmatrix}^\top
        \hessian^\n_i
        \begin{bmatrix}
            \Delta \state^\n_i \\ \Delta \coninput_i
        \end{bmatrix} 
        +
        (q_{i}^\n)^\top
        \begin{bmatrix}
            \Delta \state^\n_i \\ \Delta \coninput_i
        \end{bmatrix} \nonumber \\
        \mathrm{s.t.} \quad &\forall i \in \{0,\ldots,\horizon-1\}, 
        \forall \n \in \{1,\ldots,\numsamples \}, 
        \label{eq:SQP_cost} 
        \\
        & \Delta \state^\n_0 = 0, \\
        & \begin{aligned} 
            \Delta \state^\n_{i+1} = \> & f(\hat{\state}^\n_i,\hat{\coninput}_i) + B_d \dynGP^\n(\hat{\state}^\n_i,\hat{\coninput}_i) - \hat{\state}^\n_{i+1} \\&+ \hat{A}^\n_i \Delta \state^\n_i + \hat{B}^\n_i \Delta \coninput_i,
        \end{aligned} \\
        & 0 \geq h(\hat{\state}^\n_i,\hat{\coninput}_i) + \hat{H}^\state_{x,i} \Delta \state^\n_i + \hat{H}^\coninput_{u,i} \Delta u^\n_i, \\
        & 0 \geq h_\horizon(\hat{\state}^\n_\horizon) + \hat{H}^\state_{x,\horizon} \Delta \state^\n_\horizon.
    \end{align}
\end{subequations}
Thereby, $\hessian_i^\n$ denotes a positive definite approximation of the Hessian of the Lagrangian, $q^\n_i$, the cost gradient, and
\begin{align}
    \hat{A}^\n_i &= \frac{\partial (f + B_d g^\n)}{\partial x}, 
    \quad
    \hat{B}^\n_i = \frac{\partial (f + B_d g^\n)}{\partial u} \label{eq:SQP_ABk} \\ 
    \hat{H}^\n_{\state,i} &= \frac{\partial h}{\partial \state}, 
    \quad
    \hat{H}^\n_{\coninput,i} = \frac{\partial h}{\partial \coninput},
    \quad
    \hat{H}^\n_{\state,\horizon} = \frac{\partial h_\horizon}{\partial \state},
\end{align}
the Jacobians of the dynamics and other (in-)equality constraints with respect to the states and inputs, respectively,
each evaluated at \mbox{$(\hat{\state}^\n_i, \hat{\coninput}_i)$}.
The solution of~\eqref{eq:SafeGPMPCsqp}
in terms of 
inputs
\mbox{$\Delta \bm{\coninput} = (\Delta \coninput_0, \ldots, \Delta \coninput_{\horizon - 1})$}
and states \mbox{$\Delta \bm{\state} = \left\{ \Delta \bm{\state}^\n \right\}_{n=1}^\numsamples$}, with 
\mbox{$\Delta \bm{\state}^\n = (\Delta \state^\n_0, \ldots, \Delta \state^\n_\horizon)$}, 
is then used to update
the current solution estimate
\mbox{$\bm{\hat{\state}}^\n = (\hat{\state}^\n_0, \ldots, \hat{\state}^\n_\horizon)$}, 
\mbox{$\bm{\hat{\coninput}} = (\hat{\coninput}_0, \ldots, \hat{\coninput}_{\horizon - 1})$},
for the next iteration, 
\mbox{$\bm{\hat{\state}}^\n \leftarrow \bm{\hat{\state}}^\n + \Delta \bm{\state}^\n$}, 
\mbox{$\bm{\hat{\coninput}} \leftarrow \bm{\hat{\coninput}} + \Delta \bm{\coninput}$},
until convergence.

Evaluation of~\eqref{eq:SQP_ABk}
requires
obtaining the values of
the sampled functions~$\dynGP^\n$ and their respective Jacobians, 
i.e.,
\begin{align}
    \dynGP^\n(\hat{\state}^\n_i,\hat{\coninput}_i), && \frac{\partial g^\n}{\partial x} (\hat{\state}^\n_i,\hat{\coninput}_i), && \frac{\partial g^\n}{\partial u} (\hat{\state}^\n_i,\hat{\coninput}_i), \label{eq:SQP_needed_g_gradients}
\end{align}
at 
the
current 
solution estimate
\mbox{$(\hat{\state}^\n_i, \hat{\coninput}_i)$}. 
In the following, we discuss how to efficiently draw these function and Jacobian values jointly from a Gaussian process.

\subsection{Gaussian processes with derivatives} \label{sec:SQP_derivative_GP}

To sample function values and derivatives jointly that are consistent with sample paths of 
\mbox{$\GPtrunc[GP](0,\kernelfunc;\bar{\mathcal{D}})$},
we model the function~$\dynGPtrue$ and its spatial derivatives as a multi-output GP, with the matrix-valued kernel function \mbox{$\kernelfuncjoint: \mathbb{R}^{n_z} \times \mathbb{R}^{n_z} \rightarrow \mathbb{R}^{(n_x + 1) \times (n_x + 1)}$},
\begin{align}
    \kernelfuncjoint(\gpinp_a,\gpinp_b) = \begin{bmatrix}
        \kernelfunc(\gpinp_a,\gpinp_b) & \frac{\partial \kernelfunc(\gpinp_a,\gpinp_b)}{\partial \gpinp_a} \\
        \frac{\partial \kernelfunc(\gpinp_a,\gpinp_b)}{\partial \gpinp_b}^\top & \frac{\partial^2 \kernelfunc(\gpinp_a,\gpinp_b)}{\partial \gpinp_a \partial \gpinp_b}
    \end{bmatrix},
\end{align}
see, e.g.,~\cite[Sec.~9.4]{rasmussen_gaussian_2006}.
To simplify notation, 
note that we have dropped
the sample index $(\cdot)^\n$
in both input and output arguments.

Given the initial data set~$\bar{\mathcal{D}}$ in terms of only the function values and no derivative measurements, 
the GP with derivatives is conditioned on partial measurements.
The inference formulas are easily derived starting from
the standard GP regression case of
full function-value and derivative measurements, 
see, e.g.~\cite[Sec.~2]{rasmussen_gaussian_2006}:
Let 
the partial data 
set 
be obtained by
a linear projection
with a data selection matrix \mbox{$P \in \mathbb{N}^{(n_\state+1)\datasetdim \times \datasetdim}$},
such that 
\mbox{$Y = P^\top Y_{\mathrm{d}}$},
where
\mbox{$Y_{\mathrm{d}} \in \mathbb{R}^{(n_\state+1)\datasetdim}$}
is the (unavailable) data set covering all output dimensions 
(including derivative data for~$\dynGPtrue$).
Then, 
the posterior mean and covariance of
$\mathcal{GP}(0,\kernelfuncjoint)$ conditioned on $Y$,
evaluated at $\testinputdim$ test inputs
\mbox{$Z_\star \in \mathbb{R}^{n_\gpinp \times \testinputdim}$},
\mbox{$Z_\star = \begin{bmatrix} z_1, \ldots, z_{\testinputdim} \end{bmatrix}$},
are given as
\begin{align}
    \mujoint(Z_\star) &= \kernelfuncjoint(Z_\star, Z) P 
    \tilde{K}_{ZZ}^{-1} 
    Y, \\
    \Sigmajoint(Z_\star, Z'_\star) &= \ldots \\
    & \hspace{-0.9cm} \kernelfuncjoint(Z_\star, Z'_\star) - \kernelfuncjoint(Z_\star, Z) P
    \tilde{K}_{ZZ}^{-1} 
    P^\top \kernelfuncjoint(Z, Z'_\star), \notag
\end{align} 
see, e.g.,~\cite[Sec.~A.2]{rasmussen_gaussian_2006}, where $\left[ \kernelfuncjoint(Z,Z') \right]_{ij} \doteq \kernelfuncjoint(z_i, z'_j)$ for any matrices $Z, Z'$ with respective column vectors 
\mbox{$z_i, z'_j \in \mathbb{R}^{n_z}$},
and  
\mbox{$\tilde{K}_{ZZ} = P^\top \kernelfuncjoint(Z,Z) P  + \gpmeasvar^2 I_\datasetdim$} 
is the 
same
covariance matrix 
as in~\cref{eq:GP_posterior_covariance}.
Note that
the 
posterior mean~$\mu(Z_\star)$ and covariance~$\Sigma(Z_\star)$ of the original GP are 
recovered by projecting
\mbox{$\mu(Z_\star) = P_\star^\top \mujoint(Z_\star)$}, 
\mbox{$\Sigma(Z_\star) = P_\star^\top \Sigmajoint(Z_\star) P_\star$}, 
respectively, 
with an output selection matrix 
\mbox{$P_\star \in \mathbb{N}^{n_\star (\statedim + 1) \times n_\star}$}; 
hence,
it is easily verified that
the predictions of the GP function \emph{values} with $\mathcal{GP}(0,\kernelfuncjoint)$ and $\mathcal{GP}(0,\kernelfunc)$, given only function-value measurements~$Y$, are identical.

\subsection{Sequential GP sampling} \label{sec:SQP_sequential_sampling}

Using the derivative kernel,
we can jointly draw consistent function and derivative samples 
from the GP posterior
by computing~\cite[Sec.~A.2]{rasmussen_gaussian_2006}
\begin{align}
    \begin{aligned}
        G_\star 
        &
        \doteq
            \left[ \ldots, \dynGP(z_i), \frac{\partial g}{\partial \gpinp} (z_i), \ldots \right]^\top 
            \\
        &=
        \mujoint(Z_\star)
        + 
        \sqrt{
        \Sigmajoint
        (Z_\star, Z_\star)
        }
        \eta,    
    \end{aligned}
    \label{eq:GP_sampling}
\end{align}
where 
\mbox{$\eta \sim \mathcal{N}(0,I_{\testinputdim})$}
and
$\sqrt{\Sigma(Z_\star, Z_\star)}$ denotes the square-root of the GP posterior covariance matrix.
To ensure that the component corresponding to the function values 
is
sampled 
from the 
truncated Gaussian distribution~$\GPtrunc[N](\mu(Z_\star),\sigma^2(Z_\star);\bar{\mathcal{D}})$, 
we thereby discard all 
samples whose value-component 
violates
the confidence bounds in Lemma~\ref{thm: beta}
at any input location\footnotemark.
\footnotetext{%
    We 
    analogously
    denote by \mbox{$\GPtrunc[N](\mujoint(Z_\star), \Sigmajoint(Z_\star))$} the truncated Gaussian distribution for the function value and Jacobian, where only the value-component is truncated.
}

\begin{figure*}
    \centering
    \includegraphics*[trim=0 0 0 0, clip, width=1.0\textwidth]{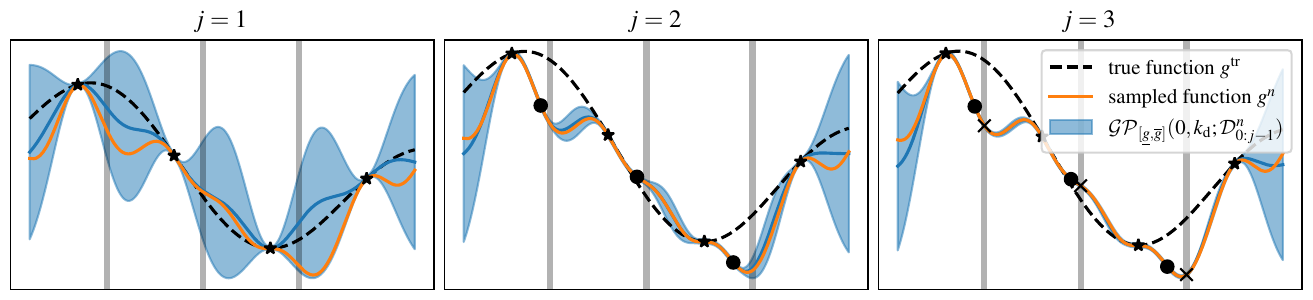}
    \caption{Conditioning on the past SQP iterations to construct a sample from GP. In each SQP iteration~$j$ in \cref{alg:sampling_gpmpc_sqp} (or forward-sampling iteration in \cref{prop:SQPsampling}), consistent function and derivative values of $g^\n$ are obtained by sampling from $\GPtrunc[GP](0, \kernelfuncjoint; \mathcal{D}^\n_{0:j-1})$.
    Initial training 
    points
    are
    denoted by black stars;
    sampled values at iteration $j \in \{1,2\}$,
    by 
    $\{ \text{circles, crosses} \}$, respectively; 
    sampling locations, by gray vertical bars.%
    }
    \label{fig:sequentialsampling}
\end{figure*}

Using~\eqref{eq:GP_sampling}, the values and derivatives of a function
\mbox{$g \sim \GPtrunc[GP](0,\kernelfunc,\bar{\mathcal{D}})$}
can be obtained at a known set of input locations. 
However, during the SQP algorithm, the same function sample is to be evaluated at different input locations, whose value is only known at each respective iteration of the algorithm.  
To ensure consistency of the sampled function values across all SQP iterations, 
we thus employ 
the iterative
sampling strategy 
proposed by~\cite[Sec.~3]{conti_gaussian_2009}.
The sampling strategy 
is based on 
a simple
equivalence relation,
allowing to 
sample 
$G_\star = \begin{bmatrix} \ldots, G_j, \ldots \end{bmatrix}$
from 
$p(G_\star | \bar{\mathcal{D}}, Z_\star)$
at all input locations
\mbox{$Z_\star = \begin{bmatrix} \ldots, Z_j, \ldots \end{bmatrix}$}
by iteratively sampling 
$G_j$
from the conditional distributions
$p(G_j | \mathcal{D}_{0:j-1}, Z_j)$,
where 
\mbox{$\mathcal{D}_0 \doteq \bar{\mathcal{D}}$} is the initial data set,
and the sets
\mbox{$\mathcal{D}_{j} = \left\{ G_{j}, Z_{j} \right\}$}, 
\mbox{$j = 1, \ldots, \numsqp$},
contain
the previously sampled values at every iteration;
see the illustration 
in \cref{fig:sequentialsampling}.
While this idea has also been applied in, e.g., \cite[Sec.~III.C]{umlauft_scenario-based_2018-1},~\cite[Sec.~IV.B]{lederer_confidence_2020},
for completeness, 
we state 
this equivalence
again
in the following lemma. 

\begin{lemma} {\upshape(cf. \cite[Property 1]{beckers_prediction_2021})}
    \label{prop:SQPsampling}
    Let 
    \mbox{$p(G_\star | \bar{\mathcal{D}}, Z_\star) = \GPtrunc(G_\star; \mu(Z_\star), \Sigma(Z_\star))$} 
    be the 
    truncated Gaussian posterior distribution
    of a Gaussian process $\mathcal{GP}(0, k; \mathcal{D}_0)$, 
    conditioned on the data set $\mathcal{D}_0 \doteq \bar{\mathcal{D}}$, 
    and evaluated at $\numsqp$ sets of input locations 
    $Z_j$, $j = 1, \ldots, \numsqp$, 
    with
    \mbox{$Z_\star = \begin{bmatrix} Z_1, \ldots, Z_\numsqp \end{bmatrix} \in \mathbb{R}^{n_z \times \horizon \numsqp}$}.
    It holds that 
    \begin{align}
        p( G_\star | \bar{\mathcal{D}}, Z_\star) = \prod_{j=1}^{\numsqp} p(G_j | \mathcal{D}_{0:j-1}, Z_j).
    \end{align}
\end{lemma}

\begin{proof}
    By the definition of the conditional probability distribution,
    it follows that the joint likelihood can be factorized, i.e., 
    \begin{align*}
        p( G_\star | \mathcal{D}_0, Z_\star) &= p(G_{2:\numsqp} | \mathcal{D}_0, G_1, Z_\star) p(G_1 | \mathcal{D}_0, Z_\star) \\
        &= p(G_{2:\numsqp} | \mathcal{D}_{0:1}, Z_{2:\numsqp}) p(G_1 | \mathcal{D}_0, Z_1) \\
        &= \ldots = \prod_{j=1}^{\numsqp} p(G_j | \mathcal{D}_{0:j-1}, Z_j). \\[-7ex]
    \end{align*}
\end{proof}

\subsection{Discussion of Sampling-based GP-MPC}
\label{sec:discussion}

The sampling-based SQP algorithm is summarized in 
\cref{alg:sampling_gpmpc_sqp}.
In the following, 
we 
comment on 
the 
algorithm's convergence
and
its properties at convergence. 
In particular, we
show that, at convergence, 
the true reachable set 
is recovered as the number of samples goes to infinity.
Last,
we
discuss its efficient 
receding-horizon
implementation.

\begin{algorithm}
    \caption{Sampling-based GP-MPC-SQP}
    \label{alg:sampling_gpmpc_sqp}
    \SetKwInOut{Input}{input}
    \SetKwProg{Fn}{Function}{:}{end}
    
    \Input{initial guess $\hat{x}^\n_i, \hat{u}_i$, data set $\mathcal{D}_0^\n \supseteq \bar{\mathcal{D}}$}
        \For{\text{$j = 1, \ldots L$ SQP iterations}}{
            \tcc{Preparation phase}
            Set $\hat{Z}^\n \leftarrow \begin{bmatrix} \ldots, (\hat{\state}^\n_i, \hat{\coninput}_i), \ldots \end{bmatrix}$\;
            Sample 
            $
            \hat{G}^\n
            \sim 
            \GPtrunc[N](\mujoint(\hat{Z}^\n),\Sigmajoint(\hat{Z}^\n);\mathcal{D}^\n_{0:j-1})$\; 
            Compute $\hat{A}^\n_i,\hat{B}^\n_i, \hat{H}^\n_{x,i}, \hat{H}^\n_{x,\horizon}, \hat{H}^\n_{u,i}$, $H_i^\n$, $q_i^\n$\; 
            Set $\mathcal{D}_j^\n \leftarrow \left\{ \hat{Z}^\n, \hat{G}^\n \right\}$\;
            \tcc{Feedback phase}
            Solve QP~\eqref{eq:SafeGPMPCsqp} for $\Delta \bm{\state}, \Delta \bm{\coninput}$\;
            Set $\{ \bm{\hat{\state}}, \bm{\hat{\coninput}} \} \leftarrow \{ \bm{\hat{\state}} + \Delta \bm{\state}, \bm{\hat{\coninput}} + \Delta \bm{\coninput} \}$\;
        }
    \Return{$\hat{x}^\n_i, \hat{u}_i, \mathcal{D}^\n_{0:\numsqp}$}
\end{algorithm}

\subsubsection{Convergence of the SQP algorithm}

Applied to 
SQP, 
Lemma~\ref{prop:SQPsampling} allows us to obtain the required values and Jacobians of each sample in~\eqref{eq:SQP_needed_g_gradients} while running the algorithm, 
\emph{as if} we had sampled the infinite-dimensional function 
\mbox{$g^\n \sim \GPtrunc[GP](0,\kernelfunc,\bar{\mathcal{D}})$} 
beforehand to evaluate and differentiate it at the
SQP iterates.
This implies that 
the convergence of the SQP iterates towards a local minimizer of the NLP~\eqref{eq:SafeGPMPCinf}
is not affected, 
but rather 
follows standard arguments for Newton-type methods, see, e.g.~\cite[Sec.~18]{nocedal_numerical_2006}.

\subsubsection{Simultaneous GP sampling and optimization}

A key challenge for
simulations with GP-SSMs 
using
forward-sampling
is
that the GP needs to be sampled (and re-conditioned) \emph{in sequence} along the simulation horizon. 
However, due to the proposed \emph{simultaneous} optimization over state and input variables in the SQP algorithm, 
at each iteration, 
function and derivative values 
can efficiently be drawn jointly and in parallel 
based on the
state and input values $(\hat{x}^\n_i, \hat{u}_i)$
at the $i$-th prediction step.
The functions 
\mbox{$g^\n \sim \GPtrunc[GP](0,\kernelfunc;\bar{\mathcal{D}})$} 
are thus efficiently sampled while the SQP algorithm iterates towards a feasible and locally optimal solution of the finite-sample problem~\eqref{eq:SafeGPMPCinf}. 
At convergence of the SQP algorithm, 
i.e., when $\Delta \bm{x} = 0$ and $\Delta \bm{u} = 0$ in~\eqref{eq:SafeGPMPCsqp},
\cref{prop:SQPsampling} ensures
that the generated trajectories 
correspond to exact simulations of the GP-SSM dynamics.

\begin{corollary}
    \label{thm:true_traj_at_convergence}
    Let $\bm{\hat{u}} = (\hat{\coninput}_0, \ldots, \hat{\coninput}_{\horizon-1})$, $\bm{\hat{x}}^\n = (\hat{x}^\n_0, \ldots, \hat{x}^\n_\horizon)$ be an input and state sequence obtained by running 
    Algorithm~\ref{alg:sampling_gpmpc_sqp} 
    until convergence.
    Then, $\bm{\hat{x}}^\n$ is equal to the state sequence generated by 
    \mbox{$\state_{0}^\n = \bar{x}_0$}, 
    \mbox{$\state_{i+1}^\n = \dynNom(\state^\n_i, \hat{\coninput}_i) + \gpsubspace \dynGP^\n(\state^\n_i,\hat{\coninput}_i)$},
    for 
    \mbox{$i=0,\ldots,\horizon-1$}, where 
    $\dynGP^\n \sim \GPtrunc[GP](0,\kernelfunc;\bar{\mathcal{D}})$. 
\end{corollary}

In addition to the joint sampling of the multivariate GP along the horizon, 
the sampling procedure can also be efficiently parallelized across the independent output dimensions and samples. 
By evaluating the $(\numsamples \times \gpdim)$ samples and output dimensions as a \emph{batch} in the \texttt{GPyTorch}~\cite{gardner_gpytorch_2018} framework,
drastic speedups are achieved on supported parallel computing architectures. 

\subsubsection{Reachable set approximation}

The next proposition establishes that, 
as the number of samples goes to infinity,
at least one of 
the sampled function values~$\dynGP^\n(z)$
will be arbitrarily close to the true dynamics~$\dynGPtrue(z)$,
given that the true dynamics is contained within the uncertainty set~$\dynSet$.

\begin{proposition}
    \label{prop:reachable_set_convergence}
    Let 
    \cref{assump:q_RKHS}
    hold.
    Let 
    $z \in \mathcal{Z}$
    and 
    \mbox{$g^\n(z) \sim \GPtrunc(G_\star; \mu(z), \sigma(z))$}
    where $\n \in [1, \numsamples]$.
    For any $\epsilon > 0$, 
    it holds that
    \begin{align}
        \lim_{\numsamples \rightarrow \infty} 
        \Pr \left(
        \min_{\n \in [1,\numsamples]} 
        \left| 
        \dynGPtrue(z) - g^\n(z) 
        \right|
        < \epsilon
        \right)
        \geq 1 - p.
    \end{align}
\end{proposition}
\begin{proof}
    \cref{thm: beta} implies that 
    \mbox{$\Pr(\dynGPtrue \in \dynSet) \geq 1-p$}. 
    Now, assume that $\dynGPtrue \in \dynSet$.
    Since 
    $\GPtrunc(G_\star; \mu(z), \sigma(z))$
    has full support on the domain~$[\lbdyn(z), \ubdyn(z)]$,
    it holds that 
    \mbox{
        $p_{\mathrm{in}} \doteq 
        \Pr(|\dynGPtrue(z) - g^\n(z)| 
        < \epsilon) 
        > 0$
        for all $n \in [1,\numsamples]$.
    } 
    Therefore, for $\numsamples \rightarrow \infty$, the probability that 
    no sample 
    is \mbox{$\epsilon$-close} to $\dynGPtrue(z)$
    is
    $\lim_{\numsamples \rightarrow \infty} (1 - p_{\mathrm{in}})^\numsamples = 0$. 
    This implies that the joint probability of 
    both events is greater or equal~$1-p$,
    which concludes the proof.
\end{proof}

\cref{prop:reachable_set_convergence} can be trivially extended to hold 
for any finite set of
input locations. 
Combined with~\cref{thm:true_traj_at_convergence},
this ensures 
that 
the sampling-based approximation 
convergences
to the true reachable set 
as the number of samples 
increases
(given a finite number of SQP steps).

\subsubsection{Efficient receding-horizon implementation}
\label{sec:receding_horizon}

Next, we discuss how to embed the proposed sampling strategy into a real-time SQP algorithm for a receding-horizon controller
that
solves~\eqref{eq:SafeGPMPCsqp} at each 
time step~$k \in \mathbb{N}$,
applying the first element of 
the optimized input sequence to the true system. 
Thereby, we focus on efficient real-time optimization rather than 
closed-loop guarantees, which are left for future work.

\begin{algorithm}
    \caption{Receding-horizon implementation}
    \label{alg:sampling_gpmpc_closedloop}
    \SetKwInOut{Input}{input}
    \Input{$x(0), \bar{\mathcal{D}}$, $\hat{x}^\n_i, \hat{u}_i$}
    Initialize $\hat{x}_0^\n \leftarrow x(0), \mathcal{D}^\n \leftarrow \bar{\mathcal{D}}$\;
    \For{$k = 0, \ldots$}{
        \tcc{Measure next state}
        Set $x_0^\n \leftarrow f(x(k), \hat{u}_0) + B_g g(x(k), \hat{u}_0)$\;
        \tcc{Run SQP}
        Set $(\bm{\hat{x}}^\n, \bm{\hat{u}}, \mathcal{D}^\n)$ 
        $\leftarrow$ \textbf{Algorithm~\ref{alg:sampling_gpmpc_sqp}}($\bm{\hat{x}}^\n, \bm{\hat{u}}, \mathcal{D}^\n$)\;
        $\rightarrow$ Apply~$\hat{u}_0$ to~\eqref{eq:system_dyn} after \texttt{Feedback phase}\;
        $\rightarrow$ Start \texttt{Preparation phase} for time $k+1$\;
        \tcc{Keep most recent $\tilde{L}$ samples}
        Set $\mathcal{D}^\n \leftarrow \left\{ \bar{\mathcal{D}}, \mathcal{D}^\n_{L-\tilde{L}:L} \right\}$\;
    }
\end{algorithm}
For real-time MPC implementations, 
running a fixed number of SQP iterations 
achieves an efficient trade-off between 
optimality of the current iterate 
and 
fast control feedback
with respect to the
currently processed
initial condition~\cite[Sec.~4.1.2]{diehl_real-time_2001}.
A particularly efficient trade-off is given by the Real-Time Iteration~(RTI)~\cite{diehl_real-time_2005},
which runs a single SQP iteration per time step that is divided into 
a preparation phase, for computing the linearization around the current iterate,
and a feedback phase, for solving the~QP. 
With 
a sufficiently high sampling rate 
and regularity of the problem, 
this allows to continuously track the solution manifold of the finite-sample problem~\eqref{eq:SafeGPMPCinf_forall} 
as it is evolving based on the changing initial condition, see, e.g.~\cite{zanelli_contraction_2019}.

Still, even when running only a fixed number of SQP iterations per time step 
during closed-loop operation, 
storing all previously sampled values of the GP would lead to a steadily increasing solve time. 
Thus, 
to retain a fixed computational workload associated with GP sampling 
at each time step, 
we propose to discard a subset of the sampled data points at each MPC iteration,
keeping only the most recent 
\mbox{$\tilde{L} \leq L$} 
sampled values\footnote{
    While this approach 
    may seem related to
    the limited-memory GP-SSMs introduced in~\cite{beckers_prediction_2021}, 
    note that, in this case, the memory limitation is expressed across MPC iterations rather than across time.
    In each MPC iteration, the functions 
    $\dynGP^\n \sim \GPtrunc[GP](0,\kernelfunc,\bar{\mathcal{D}})$ 
    implicitly generating the sampled values and Jacobians may thus change;    
    however, they are fixed during the runtime of the SQP   algorithm, 
    leading to consistent trajectories at convergence. 
} generated at each call of~\cref{alg:sampling_gpmpc_sqp}.
The complete receding-horizon control algorithm 
is 
shown in \cref{alg:sampling_gpmpc_closedloop}.

By conditioning
the GP on
the $\tilde{L}$ most recently sampled values and Jacobians at 
the previous time step
during closed-loop operation,
smoothness of the GP sample 
provides
that the value and Jacobian in a local neighborhood of the current time step will be close to the previously sampled ones.
We expect that 
standard RTI arguments 
(see, e.g.,~\cite{zanelli_contraction_2019})
carry over 
to the proposed implementation using only a fixed number of SQP steps.

\section{Simulation results}
\label{sec:simulations}
\begin{figure*}
\begin{subfigure}[b]{0.33\textwidth}
    \includegraphics[width=1.0\textwidth]{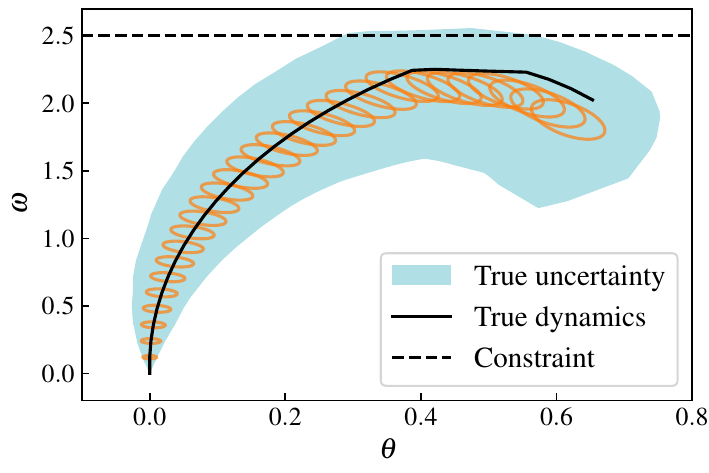}
       \caption{Cautious GP-MPC~\cite{hewing_cautious_2020}}
    \label{fig:cautius}
\end{subfigure}\hfill
\begin{subfigure}[b]{0.33\textwidth}
    \includegraphics[width=1.0\textwidth]{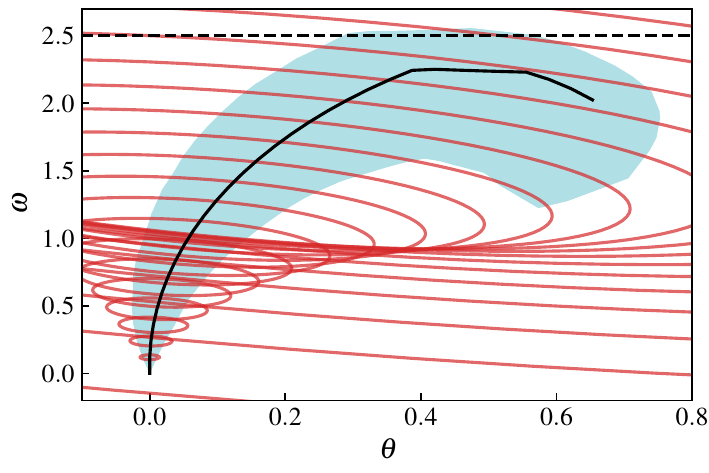}
     \caption{Robust tube-based GP-MPC~\cite{koller_learning-based_2018}}
    \label{fig:safeMPC}
\end{subfigure}
\begin{subfigure}[b]{0.33\textwidth}
    \includegraphics[width=1.0\textwidth]{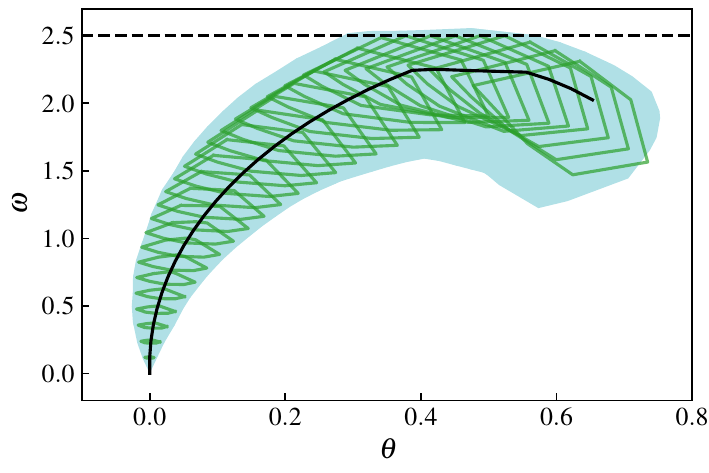}
    \caption{Robust sampling-based GP-MPC (ours)}
    \label{fig:samplingMPC}
\end{subfigure}
\caption{Comparison of uncertainty propagation for a given input sequence $\bm{u}$ by different GP-MPC methods in the pendulum example. The blue-shaded area represents the true uncertainty propagation 
for \mbox{$2\times 10^5$} dynamics sampled from $\GPtrunc[GP](0, \kernelfunc; \bar{\mathcal{D}})$, 
the solid black line
represents 
the true system, and the dashed line 
shows 
the 
angular velocity 
constraint.
Cautious GP-MPC significantly under-approximates (orange, \cref{fig:cautius}), Robust tube-based GP-MPC significantly over-approximates (red, \cref{fig:safeMPC}), whereas Robust sampling-based (proposed) GP-MPC provides a close approximation (green, \cref{fig:samplingMPC}) to the true uncertainty set, even with only 20 dynamics samples.}\label{fig:uncertainity}
\end{figure*}

In this section, we demonstrate our method 
with two 
numerical examples.
First, we show the effectiveness of the uncertainty propagation resulting from the proposed sampling-based GP-MPC method 
in comparison to existing methods~\cite{hewing_cautious_2020,koller_learning-based_2018} using a pendulum example. 
Second, we demonstrate closed-loop constraint satisfaction
under challenging constraints for a more realistic car dynamics model.

\looseness -1 In both numerical examples, we model the unknown dynamics using a separate GP for each output dimension~(cf.~\cref{remark:multidim_gp})
and set $\sqrt{\beta} = 2.5$.
A squared-exponential kernel is employed,
whose hyper-parameters
are determined using 
maximum-likelihood estimation. 
In both examples, the measurement noise 
in \cref{eq:measurements}
is set to \mbox{$\lambda^2 = 10^{-6}$}. 
For numerical stability, the same noise value is used for 
the iterative sampling strategy (\cref{sec:SQP_sequential_sampling}). 
To simplify comparisons, 
the
simulations are carried out 
without a terminal set.
\cref{alg:sampling_gpmpc_sqp} is implemented%
\footnote{%
The source code is available at \url{https://github.com/manish-pra/sampling-gpmpc}, \doi{10.3929/ethz-b-000693678}
}
using the \texttt{Python} interfaces of \texttt{acados}~\cite{verschueren_acadosmodular_2022} and~\texttt{CasADi}~\cite{andersson_casadi_2019}; for efficient GP sampling, we employ~\texttt{GPyTorch}~\cite{gardner_gpytorch_2018}.

\subsection{Comparison of uncertainty propagation using a pendulum}
We consider 
a pendulum whose 
dynamics are given by 
\begin{align*}
    \begin{bmatrix}
        \theta(\ki+1) \\
        \omega(\ki+1)
    \end{bmatrix} =     \begin{bmatrix}
        \theta(\ki) + \omega(\ki) \Delta  \\
        \omega(\ki) - g_a \sin(\theta(\ki)) \Delta / l + \alpha(\ki) \Delta 
    \end{bmatrix}. 
\end{align*}
\looseness -1 The pendulum state is $\state = [\theta, \omega]^\top$, 
where $\theta[\si{rad}]$ denotes the angular position, $\omega[\si{rad/s}]$, the angular velocity, and the control input is the angular acceleration $\alpha[\si{rad/s^2}]$. The constant parameter $l = 1\si{m}$ denotes the length of the pendulum, $\Delta=0.015 \si{s}$, the discretization time and $g_a=10 \si{m/s^2}$, the acceleration due to gravity. 
We consider the case of completely unknown dynamics, i.e., \mbox{$\dynNom(\state,\coninput) = 0$}, 
and 
residual nonlinear dynamics 
\mbox{$\dynGP: \R^3 \to \R^2$},  
\mbox{$\dynGP(\state,\coninput) = \dynGP(\theta, \omega, \alpha)$},  
and 
\mbox{$\gpsubspace = \mathbb{I}_{2 \times 2}$}.
The GP is trained using
$|\bar{\mathcal{D}}| = 45$ data points 
(including derivatives)
on an equally-spaced
$3 \times 3 \times 5$
mesh grid 
in the constraint set 
\mbox{$\mathcal{Z} = \{ (\theta, \omega, \alpha) \in [-2.14,2.14] \times [-2.5,2.5] \times[-8,8] \}$}.
The task is to control the pendulum from $x = (0,0)$ to a desired 
state of $x_{\mathrm{des}} = (2.5, 0)$. 
The cost of the MPC Problem~\eqref{eq:SafeGPMPCinf} 
is
given by
\mbox{$l_i(\theta_i, \alpha_i) = 50 (x_i - x_{\mathrm{des}})^2 + 0.1 \alpha_i^2$}, 
$\forall i =0,\ldots,\horizon-1$,
where $\horizon = 31$.

\looseness -1 We compare the uncertainty propagation of three methods in \cref{fig:uncertainity}: a) Cautious GP-MPC~\cite{hewing_cautious_2020}, which uses linearization and independence approximations, b) Robust tube-based GP-MPC~\cite{koller_learning-based_2018}, which uses ellipsoidal tube propagation and c) Robust sampling-based GP-MPC (\cref{alg:sampling_gpmpc_sqp}), which approximates the reachable set using sampled dynamics. 
For a better comparison, the uncertainty propagation is shown for the same
open-loop control sequence $\bm{\coninput}$ 
determined using 
\cref{alg:sampling_gpmpc_sqp} with $\numsamples = 20$ dynamics samples.
In \cref{fig:cautius}, the orange ellipsoids show 
the confidence sets corresponding to 
$\sqrt{\beta}$ standard deviations
for the state
using the 
Cautious GP-MPC method \cite{hewing_cautious_2020}. 
The confidence sets 
significantly under-approximate the true uncertainty (blue shaded area), 
failing to capture the true system uncertainty.
\cref{fig:safeMPC} demonstrates the uncertainty propagation computed using the Robust tube-based GP-MPC~\cite{koller_learning-based_2018} method, where the red ellipsoids show
safe sets
obtained for the same 
high-probability 
confidence set~$\dynSet$. 
Due to the exponential growth of the sets along the prediction horizon,
it significantly over-approximates the true uncertainty set, 
which would lead
to an overly conservative controller. 
\cref{fig:samplingMPC} shows convex hulls at each horizon of the 20 simulated trajectories, 
representing the uncertainty propagation of our method. 
In contrast to the other methods, 
the sampling-based approximation
provides a close approximation of the true reachable set 
(which can be further improved by using more samples), 
leading to practically meaningful constraint tightenings.

\subsection{Safe closed-loop control with car dynamics}

\looseness -1 In the following, we illustrate the performance of the proposed receding-horizon implementation for an 
overtaking
maneuver of an autonomous car,
with
particular emphasis on real-time applicability and closed-loop constraint satisfaction. 
We model the nonlinear car dynamics using a kinematic bicycle model as follows:
\begin{align*}
    \!\!\begin{bmatrix}
        \!\xpos(\ki\!+\!1)\!\!\\ \!\ypos(\ki\!+\!1)\!\\ \!\theta(\ki\!+\!1)\!\\ \!v(\ki\!+\!1)\!
    \end{bmatrix} \!\! &= \!\!\underbrace{\begin{bmatrix}
        \xpos(\ki) \\  \ypos(\ki) \\ \theta(\ki) \\ \!v(\ki) \!+ \!a(\ki) \Delta\!\!
    \end{bmatrix}}_{\eqqcolon \dynNom(\state,\coninput)} \!\! + \!\! 
    \underbrace{
        \begin{bmatrix} \mathbb{I}_{3 \times 3} \\
        0_{1 \times 3} 
    \end{bmatrix}}_{\eqqcolon \gpsubspace}
    \!\! \underbrace{\begin{bmatrix}
    \!v(\ki) \cos(\theta(\ki) \!+\! \zeta_k) \Delta \!  \\ 
    \!v(\ki) \sin(\theta(\ki) \!+ \!\zeta_k) \Delta \! \\
    \!v(\ki) \sin(\zeta_k) l_r^{-1} \Delta \! 
\end{bmatrix}}_{\eqqcolon \dynGP(\state, \coninput)} 
\label{eqn:bicycle_dyn}
\end{align*}
where \mbox{$\zeta_k =\tan^{-1}\left(\frac{l_r}{l_f + l_r} \tan(\delta(k))\right)$} is the slip angle~$[\si{rad}]$. The states $\state = [\xpos, \ypos, \theta, v]^\top$ consist of the position~$[\si{m}]$ of the car in 2D-Cartesian coordinates $[\xpos, \ypos]^\top$, its absolute heading angle~$\theta [\si{rad}]$ and longitudinal velocity~$v [\si{m/s}]$; the control input~\mbox{$\coninput = [\delta, a]^\top$}, of the steering angle~$\delta [\si{rad}]$ and linear acceleration~$a [\si{m/s^2}]$.
The distances between the car's center of
gravity to the front and the rear wheel
are denoted by
\mbox{$l_f = 1.105 [\si{m}]$} and
\mbox{$l_r = 1.738 [\si{m}]$},
respectively. 
The known part of the dynamics, $\dynNom(\state,\coninput)$, 
assumes simple integrator dynamics. 
The unknown part, \mbox{$\dynGP: \R^3 \to \R^3$}, \mbox{$\dynGP(\state, \coninput) = \dynGP(\theta, v, \delta)$}, is a nonlinear function of heading angle, velocity and steering angle.
The GP is trained with
45 prior data points
(excluding derivatives)
on an
equally spaced $3\times3\times5$ mesh grid along $(\theta, v, \delta)$ in the 
set 
$\{ (x,u) \in [-2.14,70] \times [0,6]\times[-1.14,1.14]\times[-1,15]\times[-0.6,0.6]\times[-2,2] \}$,
given by the intersection of the track, velocity and input constraints. 
In addition, the full constraint set includes
ellipsoidal obstacle avoidance constraints 
accounting for the size of the ego and other vehicles,
$({\xpos}_i - x_e)^2/9 + ({\xpos}_i - y_e)^2 \geq 5.67$,
respectively centered at the other vehicle's position $(x_e, y_e)$.

\begin{figure*}
    \centering
    \includegraphics[trim=-0.5 21.5 -0.5 21.5, clip, width=1\textwidth]{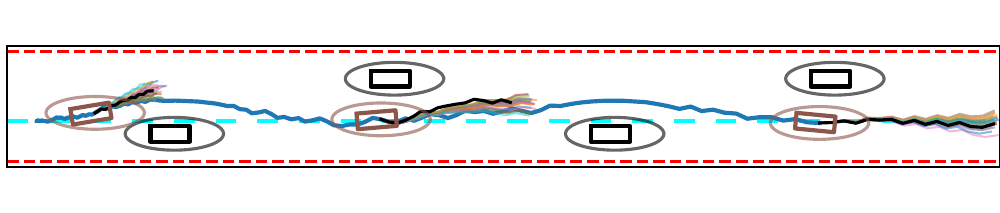}
    \caption{Demonstration of safe closed-loop control using the proposed sampling-based GP-MPC with car dynamics. The ego car (brown) successfully overtakes other vehicles (black) from left to right. The solid blue line represents the resulting closed-loop trajectory. The dashed cyan line is a reference along the y-axis, while multiple thin lines show trajectory prediction with various trajectory samples used in sampling-based GP-MPC (\cref{alg:sampling_gpmpc_closedloop}).}
    \label{fig:car_trajectory} 
\end{figure*}

\looseness -1 The task is to drive the car in the right lane 
at \mbox{$\yposref = 1.95$} %
to the target destination at 
\mbox{$\xposref = 70$}
while avoiding obstacles as shown in \cref{fig:car_trajectory}. 
Problem~\eqref{eq:SafeGPMPCinf} is solved with stage cost 
\mbox{$l_i(\state, \coninput) = \| x_i \|^2_Q + \| u_i \|^2_R $}, 
where
\mbox{$Q = \mathrm{diag}(2,36,0.07,0.005)$}
and
\mbox{$R = \mathrm{diag}(2, 2)$};
the prediction horizon is set to 
$\horizon = 16$ with sampling time 
$\Delta = 0.06 \si{s}$.
\cref{alg:sampling_gpmpc_closedloop}
is run with
\mbox{$\numsamples = 20$}
dynamics samples for $L=2$ SQP iterations.
As shown in \cref{fig:car_trajectory}, the sampling-based GP-MPC plans a safe trajectory with respect to each of the dynamics samples,
driving the car to the target destination while avoiding obstacles 
despite the uncertain dynamics.
In \cref{tab:timings}, 
we report the execution time 
per MPC solve, 
averaged over 50 time steps,
when running~\cref{alg:sampling_gpmpc_closedloop}  
on an
AMD EPYC 7543P Processor at 2.80GHz,
with different numbers of SQP iterations~$\numsqp$ and 
dynamics samples~$\numsamples$.
The GP sampling is parallelized
on an RTX 4090 GPU,
carrying over $\tilde{L} = 1$ GP samples 
between MPC iterations.
In particular, for $N=20$ dynamics and 
$L=1$ 
SQP 
iterations
real-time feasible execution times of 
32.16 ms ($\approx$ 31 Hz)
are achieved.

\newcolumntype{Y}{>{\centering\arraybackslash}X}

\begin{table*}[h] %
\centering
\caption{Average run time and standard deviation (in ms) of \cref{alg:sampling_gpmpc_sqp}  for $\horizon=16$ in the car example}
\begin{tabularx}{\textwidth}{c | Y | Y | Y | Y | Y }
$\numsamples$ $\arraybackslash$ & $\numsqp=1$ & $\numsqp=2$ & $\numsqp=3$ & $\numsqp=4$ & $\numsqp=5$ \\
\hline \hline
5  & 20.72 $\pm$ 2.07 & 42.90 $\pm$ 3.20 & 51.58 $\pm$ 18.95 & 78.13 $\pm$ 20.70 & 109.45 $\pm$ 16.78 \\
10 & 24.33 $\pm$ 4.23 & 49.75 $\pm$ 5.78 & 75.76 $\pm$ 6.20 & 103.18 $\pm$ 15.04 & 124.48 $\pm$ 21.61 \\
20 & 32.16 $\pm$ 8.77 & 63.94 $\pm$ 5.56 & 102.22 $\pm$ 12.77 & 138.82 $\pm$ 21.00 & 164.68 $\pm$ 26.96 \\
40 & 51.00 $\pm$ 16.71 & 100.23 $\pm$ 17.28 & 160.88 $\pm$ 18.27 & 237.50 $\pm$ 30.94 & 294.71 $\pm$ 38.53 \\
80 & 111.24 $\pm$ 24.72 & 234.91 $\pm$ 38.30 & 373.92 $\pm$ 42.02 & 487.02 $\pm$ 81.19 & 668.24 $\pm$ 55.64
\end{tabularx} \label{tab:timings} 
\vspace{-1.0em}
\end{table*}

\section{Conclusion}

\looseness -1 Tractable and accurate propagation of the epistemic uncertainty
poses a 
major obstacle 
for Gaussian process-based MPC.
This paper has
taken a step at addressing this challenge
by 
presenting 
a robust formulation 
for Gaussian process-based MPC,
which employs 
GP confidence bounds
to guarantee
constraint satisfaction with high probability.
For a tractable solution of the resulting GP-MPC problem,
a sampling-based approach has been proposed.
Its approximation quality, 
as well as computational tractability are
illustrated using numerical examples.
On the theoretical side,
a
promising
future research direction
is the establishment of
recursive feasibility and closed-loop constraint satisfaction guarantees for a finite number of samples.
On the computational side, 
further speedups may still be achieved by
exploiting the structure of the sampling-based dynamics in the QP solver,
as well as by employing computationally efficient GP approximations 
and tailored sampling strategies for GP models with derivatives.

\section*{Acknowledgments}
The authors would like to thank Elena Arcari and Anna Scampicchio for literature recommendations and valuable discussions.
Amon Lahr thanks 
Katrin Baumgärtner and Andrea Ghezzi
for an inspiring hackathon and discussions.

\setstretch{0.9}
\bibliographystyle{IEEEtran}
\bibliography{literature, literature_amon_edit}

\end{document}